\documentclass{article}
\usepackage{graphicx,amsmath,amsfonts,amscd,amsthm,amssymb,eucal,pstcol,psfrag,color,url}

\usepackage{tikz}
\usetikzlibrary{arrows,positioning} 
\tikzset{
    >=stealth',
    punkt/.style={
           rectangle,
           rounded corners,
           draw=black, very thick,
           text width=6.5em,
           minimum height=2em,
           text centered},
    pil/.style={
           ->,
           thick,
           shorten <=2pt,
           shorten >=2pt,}
}

\newtheorem{theorem}{Theorem}[section]

\DeclareMathOperator{\lcm}{lcm}

\begin{document}
\title{Asymptotic behavior of a series of  Euler's totient  function $\varphi(k)$
times the index of $1/k$ in a  Farey sequence}

\author{R.~Tom\'as
\thanks{rogelio.tomas@cern.ch}\\
{\centering \it CERN, CH 1211 Geneva 23, Switzerland}}

\maketitle
\begin{abstract}
Motivated by studies in accelerator physics this paper
computes the asymptotic behavior of the series $\displaystyle \sum_{k=1}^N \varphi(k) I_N\left(\frac{1}{k}\right)$,
where $\varphi(k)$ is Euler's totient function and $\displaystyle I_N\left(\frac{1}{k}\right)$ is the position
that $1/k$ occupies in the Farey sequence of order $N$.
To this end an exact formula for $\displaystyle I_N\left(\frac{1}{k}\right)$ is derived when 
all integers in $\displaystyle \left[2,\left\lceil \frac{N}{k} \right\rceil\right]$ are divisors of $N$.
\end{abstract}

\section{Results}

Let $\displaystyle I_N\left(\frac{1}{k}\right)$, with $k \in \mathbb{N}$
and $k\leq N$, be the position
that $1/k$ occupies in the Farey sequence of order $N$. Some useful facts
follow:
\begin{eqnarray}
I_N\left(\frac{0}{1}\right) &=& 1\ , \label{p1}\\
I_N\left(\frac{1}{1}\right) &=& |F_N|\ , \label{p2}\\
I_N\left(\frac{1}{k}\right) &=& 1+\sum_{j=k}^{N} \phi\left(j;\left[1, \left\lfloor\frac{j}{k}\right\rfloor \right]\right) \ , \label{p3}\\
I_N\left(\frac{1}{k}\right)&=&1+|F_N|-I_N\left(\frac{k-1}{k}\right)\  ,
\end{eqnarray}
where $|F_N|$ stands for the cardinality of 
the Farey sequence of order $N$ and Eq.~(\ref{p3}) is found at~\cite{Andrey04}~Remark~7.10 (there fractions are indexed starting with 0); $\phi(n;[\cdot,\cdot])$ is defined as the number of elements from an interval of integers that are relatively prime to $n$. Consequently $\phi(n;[1,n])\equiv\varphi(n)$ and $|F_N|=1+\sum_{j=1}^{N}\phi(j;[1,j])$.
Efficient algorithms for the computation of $I_N(x)$ were recently developed in~\cite{Jakub}.

\begin{theorem}
Given $k \in \mathbb{N}$  such that $k \leq N$, then
\begin{equation}
I_N\left(\frac{1}{k}\right)   \leq \frac{N^2 + N}{2k}\ .
\label{approx}
\end{equation}
\end{theorem}

\begin{proof}
In Eq.~(\ref{p3}) the set $\left[1, \left\lfloor\frac{j}{k}\right\rfloor\right]$
has a maximum of $j/k$ elements with $j$ running between $k$ and $N$.
Therefore $\phi$ tests a maximum of
\begin{eqnarray}
\sum_{j=k}^{N} \frac{j}{k} &=& \frac{(N-k+1)(N+k)}{2k} =\frac{N^2-k^2+N+k}{2k}
\end{eqnarray}
elements.
\end{proof}

Let the subsequence of $F_N$, $F_{N}^{1/a,\, 1/b}$ be  defined
as all the fractions of $F_N$ in $[1/a,\ 1/b]$ with $1\leq b \leq a \leq N$. 
\begin{theorem}\label{maptheo}
If $N$ is a multiple of $i$ and $i+1$ there is a bijective and order-preserving
map between $F_i$ and $F_{N}^{1/q,\, 1/(q-1)}$, with $q$ being an integer 
fulfilling  $ N/(i+1) < q \leq N/i$, given by
\begin{equation}
F_i \rightarrow F_{N}^{1/q,\, 1/(q-1)}\ , \ \ \ \ \ \ \frac{h}{k} \mapsto \frac{k}{kq-h }\label{map}\ .
\end{equation}
\end{theorem}\noindent

\begin{proof}
The demonstration  is given in two steps. The first step is shown
in Fig.~\ref{appl}, where by construction it is clear that all Farey fractions
in [$1/q,\, 1/(q-1)$] at any order are connected to the Farey fractions in $[0/1,\, 1/1]$ by the application in Eq.~(\ref{map}).
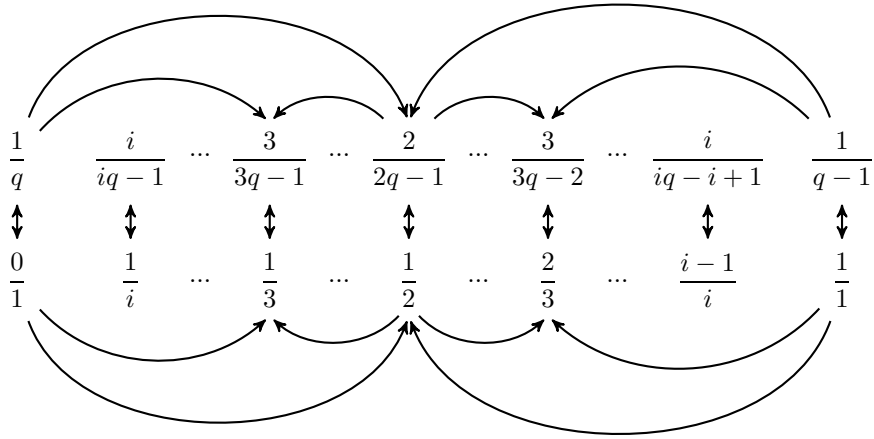
\begin{figure}
\centering
\begin{tikzpicture}[node distance=1cm, auto,]
 \node[] (1/2) {$\displaystyle \frac{1}{2}$};

\node[left=0.4cm of 1/2] (l1/2) {$...$};

\node[left=0.4cm of l1/2] (1/3) {$\displaystyle \frac{1}{3}$};

\node[right=0.4cm of l1/2] (d1/2) {$\,\,\,\,\,$};
\node[below=0.05cm of d1/2] (dd1/2) {$\,\,\,\,$}
  edge[pil,bend left=45] (1/3.south);

\node[right=0.4cm of 1/2] (r1/2) {$...$};
\node[right=0.4cm of r1/2] (2/3) {$\displaystyle \frac{2}{3}$};
 \node[left=0.4cm of 1/3] (l1/3) {$...$};
\node[left=0.4cm of l1/3] (1/i) {$\displaystyle \frac{1}{i}$};

\node[left=of 1/i] (0/1) {$\displaystyle \frac{0}{1}$}
  edge[pil,bend right=70] (1/2.south)
  edge[pil,bend right=45] (1/3.south);

\node[right=0.4cm of 2/3] (r2/3) {$...$};
\node[right=0.4cm of r2/3] (i1/i) {$\displaystyle \frac{i-1}{i}$};

\node[right=of i1/i] (1/1) {$\displaystyle \frac{1}{1}$}
   edge[pil,bend left=70] (1/2.south)
   edge[pil,bend left=45] (2/3.south);
\node[left=0.4cm of r1/2] (ddd1/2) {$\,\,\,\,\,$};
\node[below=0.05cm of ddd1/2] (dddd1/2) {$\,\,\,\,$}
  edge[pil,bend right=45] (2/3.south);
\node[above=0.6cm of 1/3] (b) {$\displaystyle \frac{3}{3q-1}$}
  edge[pil,<->, right=90] (1/3.north);
\node[above=0.6cm of 2/3] (d) {$\displaystyle \frac{3}{3q-2}$}
  edge[pil,<->, right=90] (2/3.north);
\node[above=0.6cm of 1/2] (c) {$\displaystyle \frac{2}{2q-1}$}
 edge[pil,bend right=45] (b.north)
 edge[pil,bend left=45] (d.north)
  edge[pil,<->, right=90] (1/2.north);
\node[above=0.6cm of 0/1] (a) {$\displaystyle \frac{1}{q}$}
 edge[pil,bend left=45] (b.north)
 edge[pil,bend left=70] (c.north)
  edge[pil,<->, right=90] (0/1.north);
\node[above=0.6cm of 1/1] (e) {$\displaystyle \frac{1}{q-1}$}
 edge[pil,bend right=45] (d.north)
 edge[pil,bend right=70] (c.north)
  edge[pil,<->, right=90] (1/1.north);
\node[above=0.6cm of i1/i] (f) {$\displaystyle \frac{i}{iq-i+1}$}
  edge[pil,<->, right=90] (i1/i.north);
\node[above=0.6cm of 1/i] (g) {$\displaystyle \frac{i}{iq-1}$}
  edge[pil,<->, right=90] (1/i.north);
\node[above=1.39cm of l1/2] (h) {$...$};
\node[above=1.39cm of l1/3] (i) {$...$};
\node[above=1.39cm of r1/2] (j) {$...$};
\node[above=1.39cm of r2/3] (k) {$...$};
\end{tikzpicture}
\caption{Application between the Farey fractions in $[1/q, 1/(q-1)]$ and $[0,1]$
demonstrated by using that the next Farey fraction appearing between $h/k$ and $h'/k'$
is $(h+h')/(k+k')$. By applying this rule independently to $1/q$ and $1/(q-1)$ (top)
and $0/1$ and $1/1$ (bottom) the map of Eq.~(\ref{map}) is apparent.
The second and the last to the last terms of $F_i$ (bottom) and $F_N^{1/q,\, 1/(q-1)}$ (top) are also shown. They can be computed
using Corollary~3.2 of~\cite{Matveev}.} \label{appl}
\end{figure}

A second step is needed to show that the application is bijective between
the sets $F_i$ and $F_{N}^{1/q,\, 1/(q-1)}$. The fraction $h/k$ belongs to $F_i$ if $k\leq i$, and similarly $k/(kq-h)$ belongs to $F_{N}^{1/q,\, 1/(q-1)}$ if $(kq-h) \leq N$. 
Since $N$ is a multiple of $i$ and $i+1$ and 
that $ N/(i+1) < q \leq N/i$, the largest value $(kq-h)$  takes is $\left(k\frac{N}{i}-h\right)$.
 Therefore $k/(kq-h)$ belongs to $F_{N}^{1/q,\, 1/(q-1)}$ if $k\leq i$, i.e., if  $h/k$ belongs to $F_i$. 

In the opposite direction, if $(kq-h)\leq N$ the largest possible $k$ is obtained by
inserting the smallest $q$, which is $ N/(i+1)+1$, yielding
\begin{equation}
k\left(\frac{N}{i+1}+1 \right)-h \leq N  \label{ineq}\ .
\end{equation}
If $k=h$  this corresponds to the trivial case $h/k=1/1$, which clearly
satisfies~Eq.~(\ref{appl}).
Else, $k>h$ and Eq.~(\ref{ineq}) becomes
\begin{eqnarray}
k\frac{N}{i+1}  &<& N  \label{ineq2}\ , \\
k &\leq& i  \label{ineq3}\ ,
\end{eqnarray}
concluding that if $k/(kq-h)$ belongs to $F_{N}^{1/q,\, 1/(q-1)}$, $h/k$ belongs to $F_i$.

\end{proof}


Note that in the case $q = 2$ the map in Eq.~(\ref{map}) can be viewed as a map 
from~\cite{Matveev}~Lemma 1.1 that reflects a Farey sequence to the right half sequence of a
Farey subsequence arising in the combinatorics of finite sets.

\begin{theorem}
Let $N/(i+1) \leq k \leq N/i$ and $N$ be a multiple of all integers in $[2,i]$,
then
\begin{eqnarray}
I_N\left(\frac{1}{k}\right) &=& 2+N\sum_{j=1}^{i}\frac{\varphi(j)}{j}-k\Phi(i)\ ,
\label{Itheo}
\end{eqnarray}
where $\Phi(i)$ is  the totient summatory function, 
$
\Phi(i) \equiv \sum_{j=1}^{i}\varphi(j) \equiv \left|F_i\right|-1\ 
$.
\end{theorem}
\begin{proof}
Express $I_N(1/(N/i))$  as the sum of the cardinalities of all subsequences
of the form $F_N^{1/q, 1/(q-1)}$ such that $q > N/i$,
\begin{eqnarray}
I_N\left(\frac{1}{N/i}\right)&=& 2 + \sum_{q=N/i+1}^{N} \left(\left|F_N^{1/q, 1/(q-1)}\right|-1\right)\  . \label{summing}
\end{eqnarray}
By virtue of Theorem~\ref{maptheo}, $\left|F_N^{1/q, 1/(q-1)}\right|=|F_i|$ when
$ N/(i+1) < q \leq N/i$ and Eq.~(\ref{summing}) is directly re-written 
as 
\begin{eqnarray}
I_N\left(\frac{1}{N/i}\right)&=& 2 + N\sum_{j=1}^{i-1}\frac{\Phi(j)}{j(j+1)}\ . \label{Itheo2} 
\end{eqnarray}
$I_N(1/k)$ is computed by adding to the expression above the cardinality
of the remaining subsequences between $N/i$ and $k$, yielding
\begin{eqnarray}
I_N\left(\frac{1}{k}\right)&=&\Phi(i)\left(\frac{N}{i}-k\right) + I_N\left(\frac{1}{N/i}\right)\ . \label{Itheo1}
\end{eqnarray}
After some algebra Eq.~(\ref{Itheo}) is obtained. 
\end{proof}

\begin{theorem}
\begin{eqnarray}
\sum_{k=1}^{N} \varphi(k)I_N(1/k)=\frac{N^3}{6\zeta(3)} + O\left(\frac{N^3}{\log N}\right) \label{maineq}
\end{eqnarray}
\end{theorem}
\begin{proof}

Let $N=\lcm(2,3,4,...,i_{max})$ be the least common multiple of the first $i_{max}$ numbers. The summation in the left hand side of Eq.~(\ref{maineq}) is split into two parts, 
the first part being for $k\leq N/i_{max}$. Equation~(\ref{approx}) is used
to give an upper bound to the summation corresponding to this first part.
In the second part, $k > N/i_{max}$, Eq.~(\ref{Itheo}) is used. 
The two contributions are given by
\begin{eqnarray}
\sum_{k=1}^{N/i_{max}}\varphi(k)I_N(1/k)\!\!&\leq&  (N^2+N)\sum_{k=1}^{N/i_{max}}\frac{\varphi(k)}{2k}\ ,\\
\sum_{k>N/i_{max}}^{N}\varphi(k)I_N(1/k)\!\!&=&   \sum_{i=1}^{i_{max}-1}\!\!\sum_{k=\frac{N}{i+1}+1}^{N/i}\varphi(k)\left(2+N\sum_{j=1}^{i}\frac{\varphi(j)}{j}-k\Phi(i)   \right)   \nonumber
\end{eqnarray}
After some algebra using the following relations,
\begin{eqnarray}
\sum_{k=1}^{N}\varphi(k)&=&\frac{3}{\pi^2}N^2 + O(N\log N)\ , \label{asym}\\
 \sum_{k=1}^{N} \frac{\varphi(k)}{k} &=&\frac{6}{\pi^2}N + O((\log N)^{2/3}(\log\log N)^{4/3})\ ,   \\
\sum_{k=1}^{N}\varphi(k)k&=&\frac{2}{\pi^2}N^3+O(N^2)\ ,\\
\sum_{k=1}^{\infty} \frac{\varphi(k)}{k^3}&=& \frac{\zeta(2)}{\zeta(3)}=\frac{\pi^2}{6\zeta(3)}\ ,
\end{eqnarray}
one obtains
\begin{eqnarray}
\sum_{k=1}^{N/i_{max}}\varphi(k)I_N(1/k)\!\! &\leq& \frac{3}{\pi^2i_{max}} (N^3 + N^2) \ , \label{res1.1}\\
\sum_{k>N/i_{max}}^{N}\varphi(k)I_N(1/k)\!\!&=&  \frac{N^3}{\pi^2}\sum_{i=1}^{i_{max}} \frac{\varphi(i)}{i^3}+ O(N^2\log N)\ . \label{res2}
\end{eqnarray}

When taking the limit $i_{max}\rightarrow\infty$, $N$ tends to $e^{i_{max}}$
and, inversely, $i_{max}$ tends to $\ln(N)$. 
Equations~(\ref{res1.1}) and (\ref{res2}) become
\begin{eqnarray}
\sum_{k=1}^{N/i_{max}}\varphi(k)I_N(1/k)\!\! &\leq& \frac{3}{\pi^2} \frac{N^3+N^2}{\ln(N)} \ , \\
\sum_{k>N/i_{max}}^{N}\varphi(k)I_N(1/k) &=&  \frac{N^3}{6\zeta(3)} + O(N^2\log N)\ .
\end{eqnarray}
Combining these two last equations Eq.~(\ref{maineq}) is obtained.
\end{proof}

\section{Discussion}
Let $T_N^D$ be the number of  linear, integral and irreducible polynomials of dimension $D$ of the form
\begin{equation}
\sum_{i=1}^{D} a_ix_i - a_{D+1}\ ,
\end{equation}
such that $\sum_{i=1}^{D}|a_i|\leq N$ and 
having at least one root  in the unitary $D-$cube.
$D=0$ corresponds to the null polynomial $0$.
$D=1$ corresponds to Farey sequences. 
The result of the previous section serves to
compute $T_N^2$, see~\cite{tomas,brown}.
The asymptotic behavior in $N$ of $T_N^D$, for $D=0,1,2$ follow,
\begin{eqnarray}
T_N^{0} &=& 1\ , \\
T_N^{1}&=&  \frac{3}{\pi^2}N^2 + O(N\log N)\ ,\\
T_N^{2}&=&\frac{2N^3}{3\zeta(3)} + O\left(\frac{N^3}{\log N}\right)\ .
\end{eqnarray}

\section{Acknowledgments}
Andrey Matveev is warmly thanked for making very useful suggestions on the
manuscript.


\begin{thebibliography}{20}
%
\bibitem{Andrey04} Andrey O. Matveev, {\em Relative blocking in posets}, Journal of Combinatorial Optimization, {\bf 13} (2007) 379--403. {\em Corrigendum}: {\tt arXiv:math/0411026}
%
\bibitem{Jakub}  J.~Pawlewicz and M.~P\u{a}tra\c{s}cu, {\em Order Statistics in the Farey Sequences in Sublinear Time and Counting Primitive Lattice Points in Polygons}, Algorithmica, Volume {\bf 55}, Issue 2, pp 271-282, 2009.
%
%
\bibitem{Matveev} Andrey O. Matveev, {\em Neighboring fractions in Farey subsequences},\\\noindent{\tt arXiv:math/0801.1981} 
%
\bibitem{tomas} R.~Tom\'as, \emph{From Farey sequences to resonance diagrams}
Phys. Rev. ST Accel. Beams {\bf 17}, 014001 (2014).
%
\bibitem{brown} H. Brown and K. Mahler, \emph{A Generalization of Farey Sequences: Some Exploration Via the Computer}, Journal of Number Theory, Vol. {\bf3}, 364, August 1971.
%
\end{thebibliography}
\end{document}